\documentclass{amsart}
\usepackage{url}
\usepackage{graphicx}
\usepackage{hyperref}
\hypersetup{colorlinks=true, urlcolor=blue}
\newtheorem{thm}{Theorem}[section]

\newtheorem{lem}[thm]{Lemma}
\newtheorem{prop}[thm]{Proposition}
\theoremstyle{definition}

\theoremstyle{remark}
\newtheorem{rem}[thm]{Remark}
\newtheorem{ex}[thm]{Example}

\begin{document}
\title[An algorithm to estimate the vertices of a tetrahedron]{An algorithm to estimate the vertices of a tetrahedron from uniform random points inside}

\author[A.D. V\^{\i}lcu, G.E. V\^{\i}lcu]{Alina Daniela V\^{\i}lcu, Gabriel Eduard V\^{\i}lcu}

\date{}

\abstract In this paper, we give an algorithm to infer the positions of the vertices of an unknown tetrahedron, given a sample of points which are uniformly distributed within the tetrahedron. The accuracy of the algorithm is demonstrated using some numerical experiments.
\endabstract

\maketitle

\let\thefootnote\relax
\let\thefootnote\relax

\section{Introduction}

Let $A_1,A_2,A_3,A_4$ be four non-co-planar points in three-dimensional Euclidean space $\mathbb{E}^3$, having the Cartesian coordinates
$(a_1,b_1,c_1)$, $(a_2,b_2,c_2)$, $(a_3,b_3,c_3)$ and $(a_4,b_4,c_4)$, respectively. It is well known that the volume $V$ of the tetrahedron
$A_1A_2A_3A_4$ is given by (see, e.g., \cite{SRI})
\begin{equation}\label{1}
V=\frac{1}{6}|\Delta|,
\end{equation}
where
\begin{equation}\label{2}
\Delta=
\begin{vmatrix}
a_1 & b_1 & c_1 & 1 \\
a_2 & b_2 & c_2 & 1 \\
a_3 & b_3 & c_3 & 1 \\
a_4 & b_4 & c_4 & 1
\end{vmatrix}.
\end{equation}
In what follows we denote by $I(A_1A_2A_3A_4)$ the inner of the tetrahedron $A_1A_2A_3A_4$.

Let $M_i$, $1\leq i\leq n$, be $n$ independent random points uniformly distributed inside the domain  $I(A_1A_2A_3A_4)$. We suppose that the Cartesian coordinates $(x_i,y_i,z_i)$ of the points $M_i$, $1\leq i\leq n$, are known and we propose to estimate the Cartesian coordinates of the vertices of the
tetrahedron $A_1A_2A_3A_4$. This question, that could be called \emph{tetrahedral point fitting}, generalizes the \emph{triangle fitting} of \cite{SS}: the estimation of the vertices of a triangle with uniform random points inside. We note that such types of problems are of great interest not only in computational and stochastic geometry \cite{BR,MAN,ZIN}, but also in various fields, like computer graphics, parallel computation, image analysis, biomedical imaging and signal processing \cite{GR,LJS,LJ,MM,ZW}.

The approach, based on the method of moments, leads to a solution that is effective and easy to implement efficiently. The paper is organized as follows.
In Section 2, some mixed moments of a random variable uniformly distributed within the tetrahedron are explicitly identified in terms of the coordinates of the vertices. Next, in Section 3, a method of moments estimator is proposed, equating empirical moments and theoretical moments. This leads to a system of equations for the coordinates of the estimator. Subsequently, the solutions of the system are identified and deployed to an estimation algorithm. The accuracy of the algorithm is demonstrated in Section 4 using numerical experiments, based on a procedure developed by Rocchini and Cignoni \cite{RC} for producing random points uniformly distributed in a specified tetrahedron.

\section{Mixed moments for a random point uniformly distributed inside a tetrahedron}

Let $M(X,Y,Z)$ be a random point uniformly distributed inside the domain denoted above by $I(A_1A_2A_3A_4)$. Therefore, the random vector $(X,Y,Z)$ has the probability density function $f:\mathbb{R}^3\rightarrow[0,\infty)$ defined by
\begin{equation}\label{3}
f(x,y,z)=\left\{
           \begin{array}{ll}
             \frac{6}{|\Delta|}, & \hbox{for $(x,y,z)\in I(A_1A_2A_3A_4)$} \\
             0, & \hbox{otherwise}
           \end{array}.
         \right.
\end{equation}
Then the mixed moment $(i,j,k)$ for the random vector $(X,Y,Z)\sim Unif(I(A_1A_2A_3A_4))$, denoted by $\mu_{ijk}$, is given by
\begin{equation}\label{4}
\mu_{ijk}=\mathbb{E}[X^i Y^j Z^k].
\end{equation}

\begin{lem}\label{L1}
The Jacobian of the transformation $T:(0,1)^3\rightarrow I(A_1A_2A_3A_4)$, defined by $T(u,v,w)=(x,y,z)$,
where the components of $T$ are given by
\begin{equation}\label{5}
 \left\{
          \begin{array}{ll}
           x=x(u,v,w)=(1-w)a_1+(1-v)wa_2+(1-u)vwa_3+uvwa_4 & \hbox{} \\
           y=y(u,v,w)=(1-w)b_1+(1-v)wb_2+(1-u)vwb_3+uvwb_4 & \hbox{} \\
           z=z(u,v,w)=(1-w)c_1+(1-v)wc_2+(1-u)vwc_3+uvwc_4 & \hbox{}
          \end{array}
        \right.
\end{equation}
is
\begin{equation}\label{6}
J_T=vw^2\Delta,
\end{equation}
where $\Delta$ is given by (\ref{2}).
\end{lem}
\begin{proof}
A straightforward computation shows that the Jacobian of $T$ is
\begin{eqnarray}
J_T&=&\begin{vmatrix}
\frac{\partial x}{\partial u} & \frac{\partial x}{\partial v} & \frac{\partial x}{\partial w} \\
\frac{\partial y}{\partial u} & \frac{\partial y}{\partial v} & \frac{\partial y}{\partial w}  \\
\frac{\partial z}{\partial u} & \frac{\partial z}{\partial v} & \frac{\partial z}{\partial w}
\end{vmatrix}\nonumber \\
&=&vw^2\begin{vmatrix}
a_4-a_3 & a_3-a_2  & a_2-a_1 \\
b_4-b_3 & b_3-b_2 & b_2-b_1  \\
c_4-c_3 & c_3-c_2 & c_2-c_1
\end{vmatrix}\nonumber \\
&=&vw^2\begin{vmatrix}
a_4-a_3 & a_3-a_2 & a_2-a_1 & a_1 \\
b_4-b_3 & b_3-b_2 & b_2-b_1 & b_1 \\
c_4-c_3 & c_3-c_2 & c_2-c_1 & c_1\\
0 & 0 & 0 & 1
\end{vmatrix}\nonumber \\
&=&vw^2\begin{vmatrix}
a_4 & a_3 & a_2 & a_1 \\
b_4 & b_3 & b_2 & b_1 \\
c_4 & c_3 & c_2 & c_1\\
1 & 1 & 1 & 1
\end{vmatrix}\nonumber\\
&=&vw^2\Delta.\nonumber
\end{eqnarray}

Therefore, it follows that the value of the Jacobian of the transformation $T$ is given by (\ref{6}).
\qed
\end{proof}

\begin{rem}
It is easy to see that the transformation $T$ defined by \eqref{5} maps the vertices of the unit cube to the vertices of the tetrahedron $A_1A_2A_3A_4$ as follows:
\[
(1,1,1)\mapsto A_4(a_4,b_4,c_4),
\]
\[
(0,1,1)\mapsto A_3(a_3,b_3,c_3),
\]
\[
(\clubsuit,0,1)\mapsto A_2(a_2,b_2,c_2),
\]
\[
(\clubsuit,\clubsuit,0)\mapsto A_1(a_1,b_1,c_1),
\]
where $\clubsuit$ is either 0 or 1.

In particular, we deduce that the edges of the cube are mapped to the edges of the tetrahedron (or contracted into the vertices $A_1$ and $A_2$, respectively),
because the formulas \eqref{5} for $x,y,z$ are linear in each coordinate separately.
Therefore we conclude that the faces of cube are mapped to the faces of tetrahedron (or contracted either into the vertex $A_1$ or into the edge $A_1A_2$),
due to the fat that the faces of tetrahedron are foliated by images of some line segments foliating the faces of cube. Finally, it follows that $T$ sends the interior of the unit cube onto the interior of the tetrahedron $A_1A_2A_3A_4$, as the interior of the tetrahedron is foliated by images of some line segments that foliate the interior of the unit cube.
\end{rem}

\begin{prop}\label{P1}
The random vector $(X,Y,Z)\sim Unif(I(A_1A_2A_3A_4))$ has the following mixed moments:
\begin{equation}\label{7}
\mu_{100}=\frac{1}{4}\sum_{i=1}^{4}a_i,
\end{equation}
\begin{equation}\label{8}
\mu_{010}=\frac{1}{4}\sum_{i=1}^{4}b_i,
\end{equation}
\begin{equation}\label{9}
\mu_{001}=\frac{1}{4}\sum_{i=1}^{4}c_i,
\end{equation}
\begin{equation}\label{10}
\mu_{200}=\frac{1}{10}\left(\sum_{i=1}^{4}a_i^2+\sum_{1\leq i<j\leq4}a_ia_j\right),
\end{equation}
\begin{equation}\label{11}
\mu_{020}=\frac{1}{10}\left(\sum_{i=1}^{4}b_i^2+\sum_{1\leq i<j\leq4}b_ib_j\right),
\end{equation}
\begin{equation}\label{12}
\mu_{002}=\frac{1}{10}\left(\sum_{i=1}^{4}c_i^2+\sum_{1\leq i<j\leq4}c_ic_j\right),
\end{equation}
\begin{equation}\label{13}
\mu_{300}=\frac{1}{20}\left(\sum_{i=1}^{4}a_i^3+\sum_{1\leq i<j\leq4}a_i^2a_j+\sum_{1\leq i<j<k\leq4}a_ia_ja_k\right),
\end{equation}
\begin{equation}\label{14}
\mu_{030}=\frac{1}{20}\left(\sum_{i=1}^{4}b_i^3+\sum_{1\leq i<j\leq4}b_i^2b_j+\sum_{1\leq i<j<k\leq4}b_ib_jb_k\right),
\end{equation}
\begin{equation}\label{15}
\mu_{003}=\frac{1}{20}\left(\sum_{i=1}^{4}c_i^3+\sum_{1\leq i<j\leq4}c_i^2c_j+\sum_{1\leq i<j<k\leq4}c_ic_jc_k\right),
\end{equation}
\begin{equation}\label{16}
\mu_{400}=\frac{1}{35}\left(\sum_{i=1}^{4}a_i^4+\sum_{1\leq i<j\leq4}a_i^3a_j+\sum_{1\leq i<j<k\leq4}a_i^2a_ja_k+a_1a_2a_3a_4\right),
\end{equation}
\begin{equation}\label{17}
\mu_{040}=\frac{1}{35}\left(\sum_{i=1}^{4}b_i^4+\sum_{1\leq i<j\leq4}b_i^3b_j+\sum_{1\leq i<j<k\leq4}b_i^2b_jb_k+b_1b_2b_3b_4\right),
\end{equation}
\begin{equation}\label{18}
\mu_{004}=\frac{1}{35}\left(\sum_{i=1}^{4}c_i^4+\sum_{1\leq i<j\leq4}c_i^3c_j+\sum_{1\leq i<j<k\leq4}c_i^2c_jc_k+c_1c_2c_3c_4\right),
\end{equation}
\begin{equation}\label{19}
\mu_{111}=\frac{1}{60}\left(\sum_{i=1}^{4}a_i\sum_{i=1}^{4}b_i\sum_{i=1}^{4}c_i+2\sum_{i=1}^{4}a_ib_ic_i\right).
\end{equation}
\end{prop}
\begin{proof}
From (\ref{3}) and (\ref{4}) we deduce that the mixed moment $(i,j,k)$ for the random vector $(X,Y,Z)\sim Unif(I(A_1A_2A_3A_4))$ is given by
\begin{eqnarray}
\mu_{ijk}&=&\mathbb{E}[X^i Y^j Z^k]\nonumber \\&=&\iiint\limits_{\mathbb{R}^3} x^iy^jz^kf(x,y,z)dxdydz\nonumber\\
&=&\iiint\limits_{I(A_1A_2A_3A_4)}   \frac{6x^iy^jz^k}{|\Delta|} dxdydz\nonumber.
\end{eqnarray}

Therefore, applying Lemma \ref{L1} we obtain
\begin{equation}\label{20}
\mu_{ijk}=6\int\limits_{0}^{1}\int\limits_{0}^{1}\int\limits_{0}^{1}  vw^2E_1(u,v,w)E_2(u,v,w),E_3(u,v,w)  dudvdw,
\end{equation}
where
\begin{eqnarray}
E_1(u,v,w)&=&\left[(1-w)a_1+(1-v)wa_2+(1-u)vwa_3+uvwa_4\right]^i\nonumber \\
E_2(u,v,w)&=&\left[(1-w)b_1+(1-v)wb_2+(1-u)vwb_3+uvwb_4\right]^j\nonumber\\
E_3(u,v,w)&=&\left[(1-w)c_1+(1-v)wc_2+(1-u)vwc_3+uvwc_4\right]^k\nonumber.
\end{eqnarray}

Taking now $(i,j,k)=(1,0,0)$ in (\ref{20}), we derive
\begin{eqnarray}
\mu_{100}&=&6\int\limits_{0}^{1}\int\limits_{0}^{1}\int\limits_{0}^{1}  vw^2[(1-w)a_1+(1-v)wa_2+(1-u)vwa_3+uvwa_4]  dudvdw\nonumber\\
&=&6a_1\int\limits_{0}^{1}\int\limits_{0}^{1}\int\limits_{0}^{1}v(1-w)w^2dudvdw+6a_2\int\limits_{0}^{1}\int\limits_{0}^{1}\int\limits_{0}^{1}(1-v)vw^3dudvdw\nonumber\\
&&+6a_3\int\limits_{0}^{1}\int\limits_{0}^{1}\int\limits_{0}^{1}(1-u)v^2w^3dudvdw+6a_4\int\limits_{0}^{1}\int\limits_{0}^{1}\int\limits_{0}^{1}uv^2w^3dudvdw\nonumber\\
&=&\frac{a_1+a_2+a_3+a_4}{4}\nonumber.
\end{eqnarray}

Therefore we obtain (\ref{7}) and the expressions for the relations (\ref{8})-(\ref{19}) follow in the same manner, taking different particular cases of the triple $(i,j,k)$ in the equation (\ref{20}) and using Lemma  \ref{L1}.
\end{proof}

\section{The estimation procedure}

Let $(\alpha_r, \beta_r, \gamma_r)$ denote some estimates of the Cartesian coordinates  $(a_r, b_r, c_r)$, $r\in\{1,2,3,4\}$, of the vertices of the tetrahedron $A_1A_2A_3A_4$. In this section we will apply the classical method of Moments \cite{BN,DEL,KF} to obtain the above estimates.
Therefore, we will approximate a mixed theoretical moment by its corresponding empirical moment. Accordingly, the values $(\alpha_r, \beta_r, \gamma_r)$, $r\in\{1,2,3,4\}$, are just solution of the following nonlinear system of equations
\begin{equation}\label{21}
\mu_{ijk}=\eta_{ijk},\ (i,j,k)\in M,
\end{equation}
where $\mu_{ijk}$ is the mixed moment $(i,j,k)$, $\eta_{ijk}$ is the empirical moment $(i,j,k)$ and the set $M$ is defined by the triples: $(1,0,0),(0,1,0),(0,0,1),(2,0,0),(0,2,0)$, $(0,0,2),(3,0,0),(0,3,0),(0,0,3),(4,0,0),(0,4,0),(0,0,4)$. We recall that the empirical moment $(i,j,k)$ of the quantities $(x_s,y_s,z_s)$, $1\leq s\leq n$, is given by
\begin{equation}\label{22}
\eta_{ijk}=\frac{1}{n}\sum_{s=1}^{n}x_s^iy_s^jz_s^k.
\end{equation}

Using now Proposition \ref{P1}, it is easy to see that the system of equations (\ref{21}) can be divided into three systems of equations
$S_1$, $S_2$ and $S_3$ with separable variables, as follows:
\begin{eqnarray*}\label{S1}
S_1:\ \left\{
        \begin{array}{ll}
          \displaystyle \sum_{i=1}^{4}a_i=4\eta_{100} & \hbox{} \\
           \displaystyle\sum_{i=1}^{4}a_i^2+\sum_{1\leq i<j\leq4}a_ia_j=10\eta_{200} & \hbox{} \\
          \displaystyle\sum_{i=1}^{4}a_i^3+\sum_{1\leq i<j\leq4}a_i^2a_j+\sum_{1\leq i<j<k\leq4}a_ia_ja_k=20\eta_{300} & \hbox{} \\
          \displaystyle\sum_{i=1}^{4}a_i^4+\sum_{1\leq i<j\leq4}a_i^3a_j+\sum_{1\leq i<j<k\leq4}a_i^2a_ja_k+a_1a_2a_3a_4=35\eta_{400} & \hbox{}
        \end{array},
      \right.
\end{eqnarray*}
\begin{eqnarray*}\label{S2}
S_2:\ \left\{
        \begin{array}{ll}
          \displaystyle \sum_{i=1}^{4}b_i=4\eta_{010} & \hbox{} \\
           \displaystyle\sum_{i=1}^{4}b_i^2+\sum_{1\leq i<j\leq4}b_ib_j=10\eta_{020} & \hbox{} \\
          \displaystyle\sum_{i=1}^{4}b_i^3+\sum_{1\leq i<j\leq4}b_i^2b_j+\sum_{1\leq i<j<k\leq4}b_ib_jb_k=20\eta_{030} & \hbox{} \\
          \displaystyle\sum_{i=1}^{4}b_i^4+\sum_{1\leq i<j\leq4}b_i^3b_j+\sum_{1\leq i<j<k\leq4}b_i^2b_jb_k+b_1b_2b_3b_4=35\eta_{040} & \hbox{}
        \end{array}
      \right.
\end{eqnarray*}
and
\begin{eqnarray*}\label{S3}
S_3:\ \left\{
        \begin{array}{ll}
          \displaystyle \sum_{i=1}^{4}c_i=4\eta_{001} & \hbox{} \\
           \displaystyle\sum_{i=1}^{4}c_i^2+\sum_{1\leq i<j\leq4}c_ic_j=10\eta_{002} & \hbox{} \\
          \displaystyle\sum_{i=1}^{4}c_i^3+\sum_{1\leq i<j\leq4}c_i^2c_j+\sum_{1\leq i<j<k\leq4}c_ic_jc_k=20\eta_{003} & \hbox{} \\
          \displaystyle\sum_{i=1}^{4}c_i^4+\sum_{1\leq i<j\leq4}c_i^3c_j+\sum_{1\leq i<j<k\leq4}c_i^2c_jc_k+c_1c_2c_3c_4=35\eta_{004} & \hbox{}
        \end{array}.
      \right.
\end{eqnarray*}

Now we are able to state the following result concerning the solutions of the system $S_1$.
\begin{thm}\label{T1}
An arbitrary solution $(\alpha_1,\alpha_2,\alpha_3,\alpha_4)$ of the system $S_1$ has the form $\left(z_{\sigma(1)1},z_{\sigma(2)1},z_{\sigma(3)1},z_{\sigma(4)1}\right),$
where the function $\sigma:\{1,2,3,4\}\rightarrow\{1,2,3,4\}$ is a specified permutation and $z_{\sigma(1)1},z_{\sigma(2)1},z_{\sigma(3)1},z_{\sigma(4)1}$ are just the four roots of the polynomial equation
\begin{equation}\label{23}
z^4-\lambda_{11}z^3+\lambda_{21}z^2-\lambda_{31}z+\lambda_{41}=0,
\end{equation}
where
\begin{eqnarray}\label{24}
\left\{
  \begin{array}{ll}
    \lambda_{11}=4\eta_{100} & \hbox{} \\
    \lambda_{21}=16\eta_{100}^2-10\eta_{200} & \hbox{} \\
    \lambda_{31}=64\eta_{100}^3+20\eta_{300}-80\eta_{100}\eta_{200} & \hbox{} \\
    \lambda_{41}=256\eta_{100}^4-35\eta_{400}+160\eta_{100}\eta_{300}+100\eta_{200}^2-480\eta_{100}^2\eta_{200} & \hbox{}
  \end{array}.
\right.
\end{eqnarray}
\end{thm}
\begin{proof}
Taking into account the structure of the system $S_1$ and using Vieta's formulas, we deduce that the coefficients of the polynomial equation (\ref{23}) have the form:
\begin{eqnarray*}
\lambda_{11}&=&\alpha_1+\alpha_2+\alpha_3+\alpha_4=4\eta_{100},\\
\lambda_{21}&=&\left(\displaystyle \sum_{i=1}^{4}\alpha_i\right)^2-\left(\sum_{i=1}^{4}\alpha_i^2+\sum_{1\leq i<j\leq4}\alpha_i\alpha_j\right)\\
&=&16\eta_{100}^2-10\eta_{200},\\
\lambda_{31}&=&\displaystyle\left(\sum_{i=1}^{4}\alpha_i\right)^3+\sum_{i=1}^{4}\alpha_i^3+\sum_{1\leq i<j\leq4}\alpha_i^2\alpha_j\\
&&+\sum_{1\leq i<j<k\leq4}\alpha_i\alpha_j\alpha_k-2\sum_{i=1}^{4}\alpha_i\left(\sum_{i=1}^{4}\alpha_i^2+\sum_{1\leq i<j\leq4}\alpha_i\alpha_j\right)\\
&=&64\eta_{100}^3+20\eta_{300}-80\eta_{100}\eta_{200},
\end{eqnarray*}
\begin{eqnarray*}
\lambda_{41}&=&\displaystyle\left(\sum_{i=1}^{4}\alpha_i\right)^4-\sum_{i=1}^{4}\alpha_i^4-\sum_{1\leq i<j\leq4}\alpha_i^3\alpha_j\\
&&-\sum_{1\leq i<j<k\leq4}\alpha_i^2\alpha_j\alpha_k-\alpha_1\alpha_1\alpha_3\alpha_4\\
&&+2\sum_{i=1}^{4}\alpha_i\left(\sum_{i=1}^{4}\alpha_i^3+\sum_{1\leq i<j\leq4}\alpha_i^2\alpha_j+\sum_{1\leq i<j<k\leq4}\alpha_i\alpha_j\alpha_k\right)\\
&&+\left(\sum_{i=1}^{4}\alpha_i^2+\sum_{1\leq i<j\leq4}\alpha_i\alpha_j\right)^2-3\left(\sum_{i=1}^{4}\alpha_i\right)^2\left(\sum_{i=1}^{4}\alpha_i^2+\sum_{1\leq i<j\leq4}\alpha_i\alpha_j\right)\\
&=&256\eta_{100}^4-35\eta_{400}+160\eta_{100}\eta_{300}+100\eta_{200}^2-480\eta_{100}^2\eta_{200}
\end{eqnarray*}
and the conclusion follows now easily.
\end{proof}

Similarly, we obtain the following two results concerning the solutions of the systems $S_2$ and $S_3$.

\begin{thm}\label{T2}
If $z_{12},z_{22},z_{32},z_{42}$ are the roots of the polynomial equation
\begin{equation}\label{25}
z^4-\lambda_{12}z^3+\lambda_{22}z^2-\lambda_{32}z+\lambda_{42}=0,
\end{equation}
where
\begin{eqnarray}\label{26}
\left\{
  \begin{array}{ll}
    \lambda_{12}=4\eta_{010} & \hbox{} \\
    \lambda_{22}=16\eta_{010}^2-10\eta_{020} & \hbox{} \\
    \lambda_{32}=64\eta_{010}^3+20\eta_{030}-80\eta_{010}\eta_{020} & \hbox{} \\
    \lambda_{42}=256\eta_{010}^4-35\eta_{040}+160\eta_{010}\eta_{030}+100\eta_{020}^2-480\eta_{010}^2\eta_{020} & \hbox{}
  \end{array}
\right.
\end{eqnarray}
then the vector  $\left(z_{\tau(1)2},z_{\tau(2)2},z_{\tau(3)2},z_{\tau(4)2}\right)$ is a solution of the system $S_2$, for any permutation $\tau:\{1,2,3,4\}\rightarrow\{1,2,3,4\}$.
\end{thm}

\begin{thm}\label{T3}
If $z_{13},z_{23},z_{33},z_{43}$ are the roots of the polynomial equation
\begin{equation}\label{27}
z^4-\lambda_{13}z^3+\lambda_{23}z^2-\lambda_{33}z+\lambda_{43}=0,
\end{equation}
where
\begin{eqnarray}\label{28}
\left\{
  \begin{array}{ll}
    \lambda_{13}=4\eta_{001} & \hbox{} \\
    \lambda_{23}=16\eta_{001}^2-10\eta_{002} & \hbox{} \\
    \lambda_{33}=64\eta_{001}^3+20\eta_{003}-80\eta_{001}\eta_{002} & \hbox{} \\
    \lambda_{43}=256\eta_{001}^4-35\eta_{004}+160\eta_{001}\eta_{003}+100\eta_{002}^2-480\eta_{001}^2\eta_{002} & \hbox{}
  \end{array}
\right.
\end{eqnarray}
then the vector  $\left(z_{\rho(1)3},z_{\rho(2)3},z_{\rho(3)3},z_{\rho(4)3}\right)$ is a solution of the system $S_3$, for any permutation $\rho:\{1,2,3,4\}\rightarrow\{1,2,3,4\}$.
\end{thm}

\begin{rem}
From Theorems \ref{T1}, \ref{T2} and \ref{T3} it follows that the estimates for the $x-$, $y-$ and $z-$ coordinates of the vertices $A_1,A_2,A_3,A_4$ of the tetrahedron $A_1A_2A_3A_4$ must be found in the sets  $\{z_{11},z_{21},z_{31},z_{41}\}$,  $\{z_{12},z_{22},z_{32},z_{42}\}$ and $\{z_{13},z_{23},z_{33},z_{43}\}$, respectively. Now, the problem consists in how to combine in a suitable way the roots of the equations (\ref{23}), (\ref{25}) and (\ref{27}) such that the resulting vectors are good estimates of the Cartesian coordinates of the vertices of the tetrahedron $A_1A_2A_3A_4$. In order to solve this problem, we are applying the method of moments with introduce the additional constraint
\begin{equation}\label{29}
\mu_{111}\approx\eta_{111}
\end{equation}
which the estimated vertices must satisfy.

From (\ref{19}) and (\ref{29}) we obtain the relation
\begin{equation}\label{30}
\sum_{i=1}^{4}a_i\sum_{i=1}^{4}b_i\sum_{i=1}^{4}c_i+2\sum_{i=1}^{4}a_ib_ic_i\approx60\eta_{111}.
\end{equation}

Therefore, it is necessary to determine two suitable permutations $\pi$, $\theta$ of the set $\{1,2,3,4\}$ such that the triples \[(z_{11},z_{\pi(1)2},z_{\theta(1)3}),\ (z_{21},z_{\pi(2)2},z_{\theta(2)3}),\ (z_{31},z_{\pi(3)2},z_{\theta(3)3}),\ (z_{41},z_{\pi(4)2},z_{\theta(4)3})\] obtained with the roots of the polynomial equations (\ref{23}), (\ref{25}) and (\ref{27}) to give the best approximation (\ref{30}).

Using now (\ref{24}), (\ref{26}) and (\ref{28}) in (\ref{30}), it follows that it is necessary to determine the permutations $\pi$, $\theta$
 which minimize the expression
 \begin{eqnarray}\label{31}
E(\pi,\theta)&=&|64\eta_{100}\eta_{010}\eta_{001}+2z_{11}z_{\pi(1)2}z_{\theta(1)3}+2z_{21}z_{\pi(2)2}z_{\theta(2)3}\nonumber\\
&&+2z_{31}z_{\pi(3)2}z_{\theta(3)3}+2z_{41}z_{\pi(4)2}z_{\theta(4)3}-60\eta_{111}|.
\end{eqnarray}

We note that there are $24\cdot 24$ enumerations of $\pi, \theta$.
Moreover, in order that an estimation to be valid, all the random points $M_i$, $1\leq i\leq n$, must be located inside the resulting tetrahedron. We note that this condition can be immediately verified. In fact, it is easy to see that a point $M_i(x_i,y_i,z_i)$  is within the tetrahedron of vertices $(z_{11},z_{\pi(1)2},z_{\theta(1)3})$, $(z_{21},z_{\pi(2)2},z_{\theta(2)3})$, $(z_{31},z_{\pi(3)2},z_{\theta(3)3})$, $(z_{41},z_{\pi(4)2},z_{\theta(4)3})$, if and only if the following five determinants have the same sign:
\[
\delta=\begin{vmatrix}
z_{11} & z_{\pi(1)2} & z_{\theta(1)3} & 1 \\
z_{21} & z_{\pi(2)2} & z_{\theta(2)3} & 1 \\
z_{31} & z_{\pi(3)2} & z_{\theta(3)3} & 1 \\
z_{41} & z_{\pi(4)2} & z_{\theta(4)3} & 1
\end{vmatrix},\
\delta_1=
\begin{vmatrix}
x_i & y_i & z_i & 1 \\
z_{21} & z_{\pi(2)2} & z_{\theta(2)3} & 1 \\
z_{31} & z_{\pi(3)2} & z_{\theta(3)3} & 1 \\
z_{41} & z_{\pi(4)2} & z_{\theta(4)3} & 1
\end{vmatrix},\
\delta_2=\begin{vmatrix}
z_{11} & z_{\pi(1)2} & z_{\theta(1)3} & 1 \\
x_i & y_i & z_i & 1 \\
z_{31} & z_{\pi(3)2} & z_{\theta(3)3} & 1 \\
z_{41} & z_{\pi(4)2} & z_{\theta(4)3} & 1
\end{vmatrix},
\]
\[
\delta_3=
\begin{vmatrix}
z_{11} & z_{\pi(1)2} & z_{\theta(1)3} & 1 \\
z_{21} & z_{\pi(2)2} & z_{\theta(2)3} & 1 \\
x_i & y_i & z_i & 1 \\
z_{41} & z_{\pi(4)2} & z_{\theta(4)3} & 1
\end{vmatrix},\
\delta_4=
\begin{vmatrix}
z_{11} & z_{\pi(1)2} & z_{\theta(1)3} & 1 \\
z_{21} & z_{\pi(2)2} & z_{\theta(2)3} & 1 \\
z_{31} & z_{\pi(3)2} & z_{\theta(3)3} & 1 \\
x_i & y_i & z_i & 1
\end{vmatrix}.
\]

We note that the additional constraint \eqref{29} is sufficient to correctly identify the vertices of the tetrahedron, because, even if equation \eqref{31}
would have multiple minima, only valid permutations $\pi$, $\theta$ can satisfy the conditions that all the determinants $\delta,\delta_1,\delta_2,\delta_3,\delta_4$ defined above have the same sign for all given points $M_i$, $1\leq i\leq n$, which are uniformly distributed within the tetrahedron. Otherwise, in order to uniquely identify the vertices of the tetrahedron, it would have been necessary to apply again the method of moments, equating another theoretical moment with an empirical one.

Consequently, we obtain the following algorithm to estimate the Cartesian coordinates of the vertices of the tetrahedron
$A_1A_2A_3A_4$.\\

\textbf{Algorithm $A$} (Estimate the vertices of the tetrahedron)
\begin{itemize}
  \item[Step 1.] Input: the Cartesian coordinates $(x_i,y_i,z_i)$ of the random points $M_i$, $1\leq i\leq n$, uniformly distributed inside the tetrahedron $A_1A_2A_3A_4$.
  \item[Step 2.] Compute the empirical moments $\eta_{100}$, $\eta_{010}$, $\eta_{001}$, $\eta_{200}$, $\eta_{020}$, $\eta_{002}$, $\eta_{300}$, $\eta_{030}$, $\eta_{003}$, $\eta_{400}$, $\eta_{040}$, $\eta_{004}$, $\eta_{111}$ using formula (\ref{22}).
  \item[Step 3.] Determine the roots $z_{11},z_{21},z_{31},z_{41}$ of the polynomial equation (\ref{23}).
  \item[Step 4.] Determine the roots $z_{12},z_{22},z_{32},z_{42}$ of the polynomial equation (\ref{25}).
  \item[Step 5.] Determine the roots $z_{13},z_{23},z_{33},z_{43}$ of the polynomial equation (\ref{27}).
  \item[Step 6.] Find the permutations $\pi$ and $\theta$ of the set $\{1,2,3,4\}$ which minimize the expression $E(\pi,\theta)$ given by (\ref{31}), under constraints that all the determinants $\delta,\delta_1,\delta_2,\delta_3,\delta_4$ defined above have the same sign.
  \item[Step 7.] Output: the  estimates \[(z_{11},z_{\pi(1)2},z_{\theta(1)3}),\ (z_{21},z_{\pi(2)2},z_{\theta(2)3}),\ (z_{31},z_{\pi(3)2},z_{\theta(3)3}),\ (z_{41},z_{\pi(4)2},z_{\theta(4)3})\] for the Cartesian coordinates of the vertices of the tetrahedron $A_1A_2A_3A_4$.
\end{itemize}

\end{rem}

\section{A Monte Carlo validation}

In this section we intend to demonstrate the accuracy of the Algorithm $A$ using a Monte Carlo simulation technique. For this reason we will apply a procedure  developed by Rocchini and Cignoni \cite{RC} for producing random points uniformly distributed in a specified tetrahedron, which generalizes a method given by Turk in \cite{TU} for generate random points inside a triangle. The algorithm developed by Rocchini and Cignoni generates random
points in a cube, and then folds the cube into the barycentric space of the
tetrahedron in a way that preserves uniformity, as follows.\\

\textbf{Algorithm $B$} (Generate an uniform point inside a tetrahedron - see \cite{RC})
\begin{itemize}
  \item[Step 1.] Input: the Cartesian coordinates $(a_i,b_i,c_i)$, $1\leq i\leq4$, for the vertices of the tetrahedron  $A_1A_2A_3A_4$.
  \item[Step 2.] Generate $s$, $t$, $u$ three independent random variates uniformly distributed in the interval [0,1].
  \item[Step 3.] Cut the cube $[0,1]^3$ with the plane $s+t=1$ into two triangular prisms of equal volume, and then folding all the points falling beyond the plane $s+t=1$ (i.e. in the upper prism) into the lower prism. Calculate the new $(s,t,u)$ values as follows
            \[
            (s,t,u)=\left\{
                      \begin{array}{ll}
                        (s,t,u), & \hbox{if $s+t\leq 1$;} \\
                        (1-s,1-t,u), & \hbox{if $s+t>1$.}
                      \end{array}
                    \right.
            \]
    \item[Step 4.] Cut and fold the resulting triangular prism in Step 3 with the two planes $t+u=1$ and $s+t+u=1$. Calculate the new $(s,t,u)$ values as follows
\[
(s,t,u)=\left\{
          \begin{array}{ll}
            (s,t,u), & \hbox{if $s+t+u\leq 1$;} \\
            (s,1-u,1-s-t), & \hbox{if $s+t+u>1$ and $t+u>1$;} \\
            (1-t-u,t,s+t+u-1), & \hbox{if $s+t+u>1$ and $t+u\leq1$.}
          \end{array}
        \right.
\]
    \item[Step 5.] Find the barycentric coordinates $(a,s,t,u)$ of the random point, where $(s,t,u)$  are obtained in Step 4 and $a$ is given by \[a=1-s-t-u.\]
    \item[Step 6.] Calculate the Cartesian coordinates $(x,y,z)$ of the random point $P$ with uniform distribution inside the tetrahedron $A_1A_2A_3A_4$, as follows
    \begin{eqnarray*}
    x&=&(1-s-t-u)a_1+sa_2+ta_3+ua_4,\\
    y&=&(1-s-t-u)b_1+sb_2+tb_3+ub_4,\\
    z&=&(1-s-t-u)c_1+sc_2+tc_3+uc_4.
    \end{eqnarray*}
    \item[Step 7.] Output: the Cartesian coordinates $(x,y,z)$ of the point $P$.
\end{itemize}

\begin{ex}
By using Algorithm $B$ we will generate samples  $M_i(x_i,y_i,z_i)$, $1\leq i\leq n$, of uniform random points inside the tetrahedron $A_1A_2A_3A_4$, where the Cartesian coordinates of the vertices  $A_1,A_2,A_3,A_4$ are given in \emph{Table} \ref{table:nonlin}.

\begin{table}[ht]
\caption{The Cartesian coordinates for $A_1,A_2,A_3,A_4$.} 
\centering 
\begin{tabular}{|c|c|c|c|c|c|c|c|c|c|c|c|} 
\hline 
$a_1$ & $b_1$ & $c_1$ & $a_2$ & $b_2$ & $c_2$ & $a_3$ & $b_3$ & $c_3$ & $a_4$ & $b_4$ & $c_4$\\ [0.5ex] 
\hline
1 & 5 & 1 & 2 & 3 & 7 & 3 & 4 & 2 & 4 & 2 & 6 \\ 
 [1ex] 
\hline 
\end{tabular}
\label{table:nonlin} 
\end{table}
The sample volume is $n=10000$ and the sample data file is available at the web address \url{https://gabrieleduardvilcu.files.wordpress.com/2017/04/t10000.pdf}.
Next we will apply the Algorithm $A$ in order to estimate the coordinates of the vertices of the tetrahedron.\\

\begin{itemize}
  \item[Step 1.] We consider as input data the Cartesian coordinates of the random points $M_i$, $1\leq i\leq 10000$, generated above.\\
  \item[Step 2.] Applying (\ref{22}) we obtain the empirical moments:\\

  $\eta_{100}=2.49240083718441,\ \eta_{010}=3.50450121828043,$

  $\eta_{001}=4.00141473720768,\ \eta_{200}=6.46040997915774,$

  $\eta_{020}=12.526937461275,\ \eta_{002}=17.2986562458219,$

  $\eta_{300}=17.3390669472659,\ \eta_{030}=45.6222004210248,$

  $\eta_{003}=79.5221622781012,\ \eta_{400}=48.0003817478633,$

  $\eta_{040}=169.098093705001,\ \eta_{004}=383.788813677645,$

  $\eta_{111}=33.833772770505.$\\

  \item[Step 3.] Using (\ref{24}), we derive\\

  $\lambda_{11}=9.96960334873765,\ \lambda_{21}=34.7888911395835,$

  $\lambda_{31}=49.5355349250867,\ \lambda_{41}=23.6036769044258.$ \\

  Therefore we deduce that the roots of the polynomial equation (\ref{23}) are \\

  $z_{11}=0.981455830506285,\ z_{21}=2.035278521164251,$

  $z_{31}=2.957589241805645,\ z_{41}=3.995279755261462$.\\

  \item[Step 4.] Using (\ref{26}), we obtain\\

  $\lambda_{12}=14.0180048731217,\ \lambda_{22}=71.2350860101139,$

  $\lambda_{32}=154.991087473072,\ \lambda_{42}=121.302421329732.$ \\

  Hence we derive that the roots of the polynomial equation (\ref{25}) are \\

  $z_{12}= 2.019249312445422,\ z_{22}=3.021542850832456,$

  $z_{32}= 3.972842068393069,\ z_{42}=5.004370641450761.$\\

  \item[Step 5.] Using (\ref{28}), we deduce\\

  $\lambda_{13}=16.0056589488307,\ \lambda_{23}=83.1945559280657,$

  $\lambda_{33}=153.263012709202,\ \lambda_{43}=85.0189528792398.$ \\

  Therefore we obtain that the roots of the polynomial equation (\ref{27}) are \\

  $z_{13}=0.998218500090449,\ z_{23}=2.038959216983655,$

  $z_{33}=5.961048997780328,\ z_{43}=7.007432233976292.$\\

  \item[Step 6.] The permutations $\pi$ and $\theta$ of the set $\{1,2,3,4\}$ which minimize the expression $E(\pi,\theta)$ given by (\ref{31}), subject to constraints that $\delta,\delta_1,\delta_2,\delta_3,\delta_4$ have the same sign, are
      \[
      \pi=\left(
            \begin{array}{cccc}
              1 & 2 & 3 & 4 \\
              4 & 2 & 3 & 1 \\
            \end{array}
          \right),\ \theta=\left(
            \begin{array}{cccc}
              1 & 2 & 3 & 4 \\
              1 & 4 & 2 & 3 \\
            \end{array}
          \right).
      \]

  \item[Step 7.] We obtain the following estimated values $(\alpha_1,\beta_1,\gamma_1)$,  $(\alpha_2,\beta_2,\gamma_2)$, $(\alpha_3,\beta_3,\gamma_3)$, $(\alpha_4,\beta_4,\gamma_4)$ for the Cartesian coordinates of the vertices of the tetrahedron $A_1A_2A_3A_4$:
      \begin{eqnarray}
      (\alpha_1,\beta_1,\gamma_1)&=&(z_{11},z_{\pi(1)2},z_{\theta(1)3})\nonumber\\
      &=&(z_{11},z_{42},z_{13})\nonumber\\
      &=&(0.981455830506285,5.004370641450761,0.998218500090449),\nonumber
      \end{eqnarray}
        \begin{eqnarray}
      (\alpha_2,\beta_2,\gamma_2)&=&(z_{21},z_{\pi(2)2},z_{\theta(2)3})\nonumber\\
      &=&(z_{21},z_{22},z_{43})\nonumber\\
      &=&(2.035278521164251,3.021542850832456,7.007432233976292),\nonumber
        \end{eqnarray}
          \begin{eqnarray}
      (\alpha_3,\beta_3,\gamma_3)&=&(z_{31},z_{\pi(3)2},z_{\theta(3)3})\nonumber\\
      &=&(z_{31},z_{32},z_{23})\nonumber\\
      &=&(2.957589241805645,3.972842068393069,2.038959216983655),\nonumber
        \end{eqnarray}
        \begin{eqnarray}
      (\alpha_4,\beta_4,\gamma_4)&=&(z_{41},z_{\pi(4)2},z_{\theta(4)3})\nonumber\\
      &=&(z_{41},z_{12},z_{33})\nonumber\\
      &=&(3.995279755261462,2.019249312445422,5.961048997780328).\nonumber
        \end{eqnarray}
\end{itemize}
The above estimated values are listed in \emph{Table} \ref{table:nonlin3}. Comparing now the estimates with the real coordinates in \emph{Table} \ref{table:nonlin}, we remark that the estimated values of the coordinates are enough close with the real coordinates  of the vertices of the tetrahedron $A_1A_2A_3A_4$. Therefore, the Monte Carlo procedure validates the correctness of the Algorithm $A$ proposed in the present paper.
\end{ex}

\begin{rem}
We note that, in general, larger sample sizes produce smaller standard errors, which estimate the vertices of the tetrahedron with higher precision. Hence, the accuracy increases when the sample size increases.  In order to demonstrate this, by using Algorithm $B$, we can generate samples of uniform random points inside the same tetrahedron $A_1A_2A_3A_4$ whose vertices are given in \emph{Table} \ref{table:nonlin}, using different sample sizes. We apply again the Algorithm $A$ in order to estimate the coordinates of the vertices of the tetrahedron in each case. Some results are listed in \emph{Table} \ref{table:nonlin3} and \emph{Table} \ref{table:nonlin4} (the sample data files,  as well as the resulting quartic polynomials and their roots are available at the web address  \url{https://gabrieleduardvilcu.files.wordpress.com/2017/04/s1.pdf}).

\begin{table}[ht]
\caption{Estimation of the Cartesian coordinates of the vertices $A_1,A_2,A_3,A_4$ of the tetrahedron, when the sample size is $n\in\{1000,10000,30000,50000\}$.} 
\centering 
\begin{tabular}{ |c|c|c|c|c|  }
 \hline
& $n=1000$ & $n=10000$  & $n=30000$ & $n=50000$ \\
 \hline
$\alpha_1$ & 0.900169004177569 & 0.981455830506285& 1.000711014522624 & 1.000232625196211 \\
$\beta_1$  & 5.060437422454326 & 5.004370641450761 & 4.997904587623505 & 5.000283898674374\\
$\gamma_1$ & 0.999636598543183 & 0.998218500090449 & 0.996571940347122 & 0.999408477006355 \\
\hline
$\alpha_2$ & 2.063651216768535 & 2.035278521164251 & 1.994798523350653 & 1.996696865691069 \\
$\beta_2$ & 2.956888633281736 & 3.021542850832456 & 3.01670615938827 & 3.003518292061326 \\
$\gamma_2$ & 7.005931349261772 & 7.007432233976292 & 6.982288202530113 & 6.967786607483619 \\
\hline
$\alpha_3$& 3.001229014923811 & 2.957589241805645 & 2.987003844753415& 2.994433366420113\\
$\beta_3$& 3.983736705624392 & 3.972842068393069 &  3.995224814684614 & 3.998343042049299 \\
$\gamma_3$& 1.937576328362511 & 2.038959216983655 & 2.02606593218243 & 2.00143920258018 \\
\hline
$\alpha_4$& 3.989685187967973 & 3.995279755261462 & 4.005368172726129 & 3.997478513641908 \\
$\beta_4$& 2.031101003509045  & 2.019249312445422 & 2.002717025163117 & 2.007416732186097 \\
$\gamma_4$& 6.011872224505431 & 5.961048997780328 & 5.976385993102144 & 6.01901047801575 \\
 \hline
\end{tabular}
\label{table:nonlin3} 
\end{table}

\begin{table}[ht]
\caption{Estimation of the Cartesian coordinates of the vertices $A_1,A_2,A_3,A_4$ of the tetrahedron, when the sample size is $n\in\{70000,90000,110000,130000\}$.} 
\centering 
\begin{tabular}{ |c|c|c|c|c|  }
 \hline
 & $n=70000$ & $n=90000$ & $n=110000$ & $n=130000$\\
 \hline
$\alpha_1$ & 1.003105865295161 & 1.004527508419694 & 1.005908384957321 &   1.004563178701584 \\
$\beta_1$  & 4.99550480505 & 4.996572933666054 & 4.996874637244825 & 4.997085464604526\\
$\gamma_1$ & 1.006973943599736 & 1.005569625063884 &  1.005525834804714 &  1.005443736156998\\
\hline
$\alpha_2$ &   1.997899801933102 & 1.994761292346172 & 1.996505871342539 & 1.997080967404208\\
$\beta_2$ & 3.00924729807257 & 3.006009800104706 & 3.003001585321237 &  3.001073196705324\\
$\gamma_2$ & 6.983912371495364 & 6.995712237697553 & 7.000067186843735 &   6.997260635821995\\
\hline
$\alpha_3$ & 2.992415975197417 & 2.991612914024596 & 2.992789808210174 & 2.994439000075413 \\
$\beta_3$ & 3.996354421254385 & 3.997515526964327  & 3.997543539027272 & 3.998591070496656\\
$\gamma_3$ &  2.011169554063575 & 2.008509722164336 &   2.00122548779379 &  2.002206053742553\\
\hline
$\alpha_4$ &  4.000763964265582 & 4.001485378474737 & 4.000474073657129 & 4.001167686028979 \\
$\beta_4$ & 2.00318031010065 & 2.003537621477507 & 2.004098062824868 &  2.003664771868191\\
$\gamma_4$ & 5.995614482020844 & 5.991712343275934 & 5.995947427021777 &  5.998250617490456\\
 \hline
\end{tabular}
\label{table:nonlin4} 
\end{table}

\begin{table}[ht]
\caption{Calculated standard error of the estimate for selected sample sizes.} 
\centering 
\begin{tabular}{ |c|c|  }
 \hline
  \emph{Sample size} ($n$) & \emph{Standard error of the estimate}($\sigma_{est}$) \\
 \hline
 1000 & 0.045590228800062 \\
 10000 & 0.025971318494209\\
 30000 & 0.01323336336239 \\
 50000 & 0.011254793987414 \\
 70000 & 0.007375192691978 \\
 90000 & 0.005609129893072 \\
 110000 & 0.003972822443532 \\
 130000 & 0.003313593928035 \\
 \hline
\end{tabular}
\label{table:nonlin5} 
\end{table}

\begin{figure}
 \includegraphics[width=0.75\textwidth]{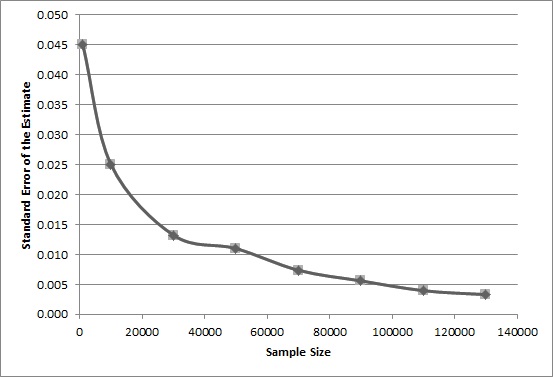}
\caption{Relationship between sample size and standard error of the estimate}
\label{fig:1}       
\end{figure}

The corresponding standard error of the estimate for different sample sizes, computed as
\[\sigma_{est}=\sqrt{\frac{\displaystyle\sum_{i=1}^{4}[(a_i-\alpha_i)^2+(b_i-\beta_i)^2+(c_i-\gamma_i)^2]}{12}},\]
 is given in \emph{Table} \ref{table:nonlin5}, while the effect of sample size on the standard error is plotted in \emph{Fig.} \ref{fig:1}. From figure, one can observe that the amount by which the standard error of the estimate decreases is most substantial between samples sizes less than 70000. In contrast, standard error of the estimate does not substantially decrease at sample sizes above 70000.
\end{rem}

\section*{Acknowledgments}
The second author was supported by a grant of Ministry of Research and Innovation, CNCS-UEFISCDI, project number PN-III-P4-ID-PCE-2016-0065, within PNCDI III.

Alina Daniela V\^{I}LCU\\
      Petroleum-Gas University of Ploie\c sti,\\
      Department of Information Technology, Mathematics and Physics,\\
      Bulevardul Bucure\c sti, Nr. 39,
      Ploie\c sti 100680-ROMANIA\\
      e-mail: daniela.vilcu@upg-ploiesti.ro\\

Gabriel Eduard V\^{I}LCU$^{1,2}$ \\
      $^1$University of Bucharest, Faculty of Mathematics and Computer Science,\\
      Research Center in Geometry, Topology and Algebra,\\
      Str. Academiei, Nr. 14, Sector 1,
      Bucure\c sti 70109-ROMANIA\\
      e-mail: gvilcu@gta.math.unibuc.ro\\
      $^1$Petroleum-Gas University of Ploie\c sti,\\
      Faculty of Economic Sciences,\\
      Bulevardul Bucure\c sti, Nr. 39,
      Ploie\c sti 100680-ROMANIA\\
      e-mail: gvilcu@upg-ploiesti.ro\\
\end{document}